\documentclass[letterpaper,11pt,final]{article}%
\usepackage{titlesec,titletoc}
\usepackage{version}
\usepackage{natbib}
\usepackage{amscd}
\usepackage[overload]{textcase}
\usepackage{makeidx}
\usepackage{etex}
\usepackage{amsxtra}
\usepackage{mathptmx}
\usepackage{graphicx}
\usepackage[amsmath,hyperref,thmmarks]{ntheorem}
\usepackage{natbib}
\usepackage{amsmath}
\usepackage{amsfonts}
\usepackage{amssymb}
\usepackage{natbib}
\theoremnumbering{arabic}
\theoremheaderfont{\scshape}
\RequirePackage{latexsym}
\theorembodyfont{\slshape}
\theoremseparator{.}
\newtheorem{X}{X}

\newtheorem{corollary}[X]{Corollary}

\newtheorem{lemma}[X]{Lemma}
\newtheorem{proposition}[X]{Proposition}

\newtheorem{theorem}[X]{Theorem}
\theorembodyfont{\upshape}

\newtheorem{remark}[X]{Remark}
\theorembodyfont{\small}

\newtheorem*{nt}{Notes}
\theorembodyfont{\normalsize}
\theoremstyle{nonumberplain}
\theoremsymbol{\ensuremath{_\Box}}
\newtheorem{proof}{Proof}

\qedsymbol{\ensuremath{_\Box}}
\newcommand{\bquote}{\begin{quote}}\newcommand\equote{\end{quote}}
\newcommand{\tstyle}{\textstyle}

\DeclareMathOperator{\End}{End}
\DeclareMathOperator{\Spec}{Spec}
\DeclareMathOperator{\Coker}{Coker}
\DeclareMathOperator{\Ker}{Ker}
\DeclareMathOperator{\Gal}{Gal}

\usepackage{titlesec}                                                                            
\titleformat*{\section}{\LARGE\bfseries}
\titleformat*{\subsection}{\Large\itshape}
\titleformat*{\subsubsection}{\scshape}
\titleformat*{\paragraph}{\itshape}

\usepackage{enumitem}
\setitemize[1]{label=$\diamond$\hspace{0.07in}}
\setlist{nolistsep}

\usepackage{array}

\usepackage{tikz}
\usetikzlibrary{matrix,arrows,positioning,decorations.pathmorphing}
\usepackage{tikz-cd}

\usepackage{wrapfig}

\usepackage{etex}

\usepackage{url}
\usepackage{verbatim}

\usepackage[notcite,notref]{showkeys}

\usepackage{mathtools}

\usepackage[overload]{textcase}

\newcommand{\eb}[1]{{\itshape\bfseries#1}}
\renewcommand{\emph}{\eb}

\usepackage{natbib}

\usepackage[textwidth=5.5in,textheight=9.0in,centering]{geometry}

\begin{document}

\title{Addendum to: Values of zeta functions of varieties over finite fields, Amer.
J. Math. 108, (1986), 297-360.}
\author{J. S. Milne}
\date{October 1, 2013}
\maketitle

The original article expressed the special values of the zeta function of a
variety over a finite field in terms of the $\widehat{\mathbb{Z}}$-cohomology
of the variety. As the article was being completed,\footnote{It was submitted
in September 1983. This addendum was originally posted on the author's website
in 2009.} Lichtenbaum conjectured the existence of complexes of sheaves
$\mathbb{Z}{}(r)$ extending the sequence $\mathbb{Z}{}$, $\mathbb{G}_{m}%
[-1]$,\ldots. The complexes given by Bloch's higher Chow groups are known to
satisfy most of the axioms for $\mathbb{Z}{}(r)$. Using Lichtenbaum's
Weil-\'{e}tale topology, we can now give a beautiful restatement of the main
theorem of the original article in terms of $\mathbb{Z}$-cohomology groups.

\subsection{Notations}

\medskip We use the notations of \cite{milne1986v}. For example,
\[
M^{(n)}=M/nM,\quad TM=\varprojlim_{n}\Ker(n\colon M\rightarrow M),\quad
z(f)=\frac{[\Ker(f)]}{[\Coker(f)]},
\]
and $\nu_{s}(r)$ denotes the sheaf of logarithmic de Rham-Witt differentials
on $X_{\text{\'{e}t}}$ (ibid., p.\,307). The symbol $l$ denotes a prime
number, possibly $p$.\qquad

\subsection{Review of abelian groups}

In this subsection, we review some elementary results on abelian groups. An
abelian group $N$ is said to be \emph{bounded} if $nN=0$ for some $n\geq1$,
and a subgroup $M$ of $N$ is \emph{pure} if $M\cap mN=mM$ for all $n\geq1$.

\begin{lemma}
\label{z1}(a) Every bounded abelian group is a direct sum of cyclic groups.

(b) Every bounded pure subgroup $M$ of an abelian group $N$ is a direct
summand of $N$.
\end{lemma}

\begin{proof}
(a) \cite{fuchs1970}, 17.2.

(b) \cite{kaplansky1954}, Theorem 7, p.\,18, or \cite{fuchs1970}, 27.5.
\end{proof}

\begin{lemma}
\label{z2}Let $M$ be a subgroup of $N,$ and let $l^{n}$ be a prime power. If
$M\cap l^{n}N=0$ and $M$ is maximal among the subgroups with this property,
then $M$ is a direct summand of $N$.
\end{lemma}

\begin{proof}
The subgroup $M$ is bounded because $l^{n}M\subset M\cap l^{n}N=0$. To prove
that it is pure, one shows by induction on $r\geq0$ that $M\cap l^{r}N\subset
l^{r}M$. See \cite{fuchs1970}, 27.7.
\end{proof}

\begin{nt}
(\cite{fuchs1970}, 9.8.) Let $B$ and $C$ be subgroups of an abelian group $A$.
Assume that $C\cap B=0$ and that $C$ is maximal among the subgroups of $A$
with this property. Let $a\in A$. If $pa\in C$ ($p$ prime), then $a\in B+C$.

Proof: We may suppose that $a\notin C$. Then $\langle C,a\rangle$ contains a
nonzero element $b$ of $B$, say, $b=c+ma$ with $c\in C$ and $m\in\mathbb{Z}{}%
$. Here $(m,p)=1$ because otherwise $b=c+m_{0}(pa)\in B\cap C=0.$ Thus
$rm+sp=1$ for some $r,s\in\mathbb{Z}{}$, and%
\[
a=r(ma)+s(pa)=rb-rc+s(pa)\in B+C.
\]

(\cite{fuchs1970}, 27.7). We prove that $M\cap l^{r}N\subset l^{r}M$ for all
$r\geq0$. This being trivially true for $r=0$, we may apply induction on $r$.
Let $m=l^{r+1}a\neq0$, $m\in M$, $a\in N$. Then $r\leq n-1$, because otherwise
$l^{r+1}a\in M\cap l^{n}N=0$. By (9.8), $l^{r}a\in l^{n}N+M$, say,
$l^{r}a=l^{n}c+d$ with $c\in N$, $d\in M$. Then $d=l^{r}a-l^{n}c\in M\cap
l^{r}N$, which equals $l^{r}M$ by the induction hypothesis. From
$m=l^{r+1}a=l^{n+1}c+ld$, we find that $(m-ld)\in M\cap l^{n+1}N=0$, and so
$m=ld\in l^{r+1}M$.
\end{nt}

\noindent Every abelian group $M$ contains a largest divisible subgroup
$M_{\mathrm{div}}$, which is obviously contained the first Ulm subgroup of
$M$, $U(M)\overset{\textup{{\tiny def}}}{=}\bigcap_{n\geq1}nM$. Note that
$U(M/U(M))=0$.

\begin{nt}
A sum of divisible subgroups is obviously divisible. For the last statement,
let $x\in M$ map to the first Ulm subgroup of $M/U(M)$. Then, for each
$n\geq1$, there exists a $y\in M$ such that $ny-x\in U(M)$, and so
$ny-x=ny^{\prime}$ for some $y^{\prime}\in M$. Now $x=n(y-y^{\prime})$, and so
$x$ is divisible by $n$ in $M$, i.e., $x\in U(M)$.
\end{nt}

\begin{proposition}
\label{z3}If $M/nM$ is finite for all $n\geq1$, then $U(M)=M_{\mathrm{div}}$.
\end{proposition}

\begin{proof}
(Cf. \cite{milne1988}, 3.3.) If $U(M)$ is not divisible, then there exists a
prime $l$ such that $U(M)\neq lU(M)$. Fix such an $l$, and let $x\in
U(M)\smallsetminus lU(M)$. For each $n\geq1$, there exists an element $x_{n}$
of $M$ such that $l^{n}x_{n}=x$. In fact $x_{n}$ has order exactly $l^{n}$ in
$M/U(M)$, and so $M/U(M)$ contains elements of arbitrary high $l$-power order.

Let $S$ be a finite $l$-subgroup of $M/U(M)$. As $U(M/U(M))=0$ and $S$ is
finite, there exists an $n$ such that $S\cap l^{n}(M/U(M))=0$. By Zorn's
lemma, there exists a subgroup $N$ of $M/U(M)$ that is maximal among those
satisfying (a) $N\supset S$ and (b) $N\cap l^{n}(M/U(M))=0$. Moreover, $N$ is
maximal with respect to (b) alone. Therefore $N$ is a direct summand of
$M/U(M)$ (Lemma \ref{z2}). As $N$ is bounded (in fact, $l^{n}N=0$), it is a
direct sum of cyclic groups (Lemma \ref{z1}). We conclude that $S$ is
contained in a finite $l$-subgroup $S^{\prime}$ of $M/U(M)$ that is a direct
summand of $M/U(M)$. Note that%
\[
S^{\prime(l)}\hookrightarrow(M/U(M))^{(l)}\simeq M^{(l)},
\]
and so $\dim_{\mathbb{F}_{l}}M^{(l)}\geq\dim_{\mathbb{F}{}_{l}}S^{\prime(l)}$.
But is clear (from the first paragraph) that $\dim_{\mathbb{F}{}_{l}}%
S^{\prime(l)}$ is unbounded, and so this contradicts the hypothesis on $M$.
\end{proof}

\begin{nt}
Cf. Fuchs, Vol II, 65.1.
\end{nt}

\begin{corollary}
\label{z4}If $TM=0$ and all quotients $M/nM$ are finite, then $U(M)$ is
uniquely divisible (= divisible and torsion-free = a $\mathbb{Q}$-vector space).
\end{corollary}

\begin{proof}
The first condition implies that $M_{\mathrm{div}}$ is torsion-free, and the
second that $U(M)=M_{\mathrm{div}}$.
\end{proof}

For an abelian group $M$, we let $M_{l}$ denote the completion of $M$ with
respect to the $l$-adic topology. Every continuous homomorphism from $M$ into
a complete separated group factors uniquely through $M_{l}$. In particular,
the quotient maps $M\rightarrow M/l^{n}M$ extend to homomorphisms
$M_{l}\rightarrow M/l^{n}M$, and these induce an isomorphism $M_{l}%
\rightarrow\varprojlim_{n}M/l^{n}M$. The kernel of $M\rightarrow M_{l}$ is
$\bigcap\nolimits_{n}l^{n}M$. See \cite{fuchs1970}, \S 13.

\begin{lemma}
\label{z5}Let $N$ be a torsion-free abelian group. If $N/lN$ is finite, then
the $l$-adic completion of $N$ is a free finitely generated $\mathbb{Z}{}_{l}$-module.
\end{lemma}

\begin{proof}
Let $y_{1},\ldots,y_{r}$ be elements of $N$ that form a basis for $N/lN$. Then%
\[
N=\tstyle\sum\mathbb{Z}{}y_{i}+lN=\tstyle\sum\mathbb{Z}{}y_{i}+l(\tstyle\sum
\mathbb{Z}{}y_{i}+lN)=\cdots=\tstyle\sum\mathbb{Z}{}y_{i}+l^{n}N,
\]
and so $y_{1},\ldots,y_{r}$ generate $N/l^{n}N$. As $N/l^{n}N$ has order
$l^{nr}$, it is in fact a free $\mathbb{Z}{}/l^{n}\mathbb{Z}{}$-module with
basis $\{y_{1},\ldots,y_{r}\}$. Let $a\in N_{l}$, and let $a_{n}$ be the image
of $a$ in $N/l^{n+1}N$. Then%
\[
a_{n}=c_{n,1}y_{1}+\cdots+c_{n,r}y_{r}%
\]
for some $c_{n,i}\in\mathbb{Z}{}/l^{n+1}\mathbb{Z}{}$. As $a_{n}$ maps to
$a_{n-1}$ in $N/l^{n}N$ and the $c_{n,i}$ are unique, $c_{n,i}$ maps to
$c_{n-1,i}$ in $\mathbb{Z}{}/l^{n}\mathbb{Z}{}$. Hence $(c_{n,i}%
)_{n\in\mathbb{N}{}}\in\mathbb{Z}{}_{l}$, and it follows that $\{y_{1}%
,\ldots,y_{r}\}$ is a basis for $N_{l}$ as a $\mathbb{Z}{}_{l}$-module.
\end{proof}

\begin{proposition}
\label{z6}Let $\phi\colon M\times N\rightarrow\mathbb{Z}{}$ be a bi-additive
pairing of abelian groups whose extension $\phi_{l}\colon M_{l}\times
N_{l}\rightarrow\mathbb{Z}{}_{l}$ to the $l$-adic completions has trivial left
kernel. If $N/lN$ is finite and $\bigcap\nolimits_{n}l^{n}M=0$, then $M$ is
free and finitely generated.
\end{proposition}

\begin{proof}
We may suppose that $N$ is torsion-free. As $\bigcap\nolimits_{n}l^{n}M=0$,
the map $M\rightarrow M_{l}$ is injective. Choose elements $y_{1},\ldots
,y_{r}$ of $N$ that form a basis for $N/lN$. According to the proof of Lemma
\ref{z5}, they form a basis for $N_{l}$ as a $\mathbb{Z}{}_{l}$-module.
Consider the map%
\[
x\mapsto(\phi(x,y_{1}),\ldots,\phi(x,y_{r}))\colon M\rightarrow\mathbb{Z}%
{}^{r}.
\]
If $x$ is in the kernel of this map, then $\phi_{l}(x,y)=0$ for all $y\in
N_{l}$, and so $x=0$. Therefore the map $M$ injects into $\mathbb{Z}{}^{r}$,
which completes the proof.
\end{proof}

\subsection{Review of Bloch's complex}

Let $X$ be a smooth variety over a field $k$. We take $\mathbb{Z}{}(r)$ to be
the complex of sheaves on $X$ defined by Bloch's higher Chow groups. For the
definition of Bloch's complex, and a review of its basic properties, we refer
the reader to the survey article \cite{geisser2005}.

The properties of $\mathbb{Z}{}(r)$ that we shall need are the following.

\begin{description}
\item[(A)$_{n_{0}}$] For all integers $n_{0}$ prime to the characteristic of
$k$, the cycle class map%
\[
\left(  \mathbb{Z}{}(r)\overset{n_{0}}{\longrightarrow}\mathbb{Z}{}(r)\right)
\rightarrow\mu_{n_{0}}^{\otimes r}[0]
\]
is a quasi-isomorphism (\cite{geisserL2001}, 1.5).

\item[(A)$_{p}$] For all integers $s\geq1$, the cycle class map%
\[
\left(  \mathbb{Z}{}(r)\overset{p^{s}}{\longrightarrow}\mathbb{Z}{}(r)\right)
\rightarrow\nu_{s}(r)[-r-1]
\]
is a quasi-isomorphism (\cite{geisserL2000}, Theorem 8.5).

\item[(B)] There exists a cycle class map $\mathrm{CH}^{r}(X)\rightarrow
H^{2r}(X_{\text{\'{e}t}},\mathbb{Z}{}(r))$ compatible (via (A)) with the cycle
class map into $H^{2r}(X_{\text{\'{e}t}},\widehat{\mathbb{Z}}(r)){}$. Here
$\mathrm{CH}^{r}(X)$ is the Chow group.

\item[(C)] There exist pairings%
\[
\mathbb{Z}{}(r)\otimes^{L}\mathbb{Z}{}(s)\rightarrow\mathbb{Z}{}(r+s)
\]
compatible (via (A)$_{n}$) with the natural pairings%
\[
\mu_{n}^{\otimes r}\times\mu_{n}^{\otimes s}\rightarrow\mu_{n}^{\otimes
r+s}\,,\quad\text{gcd}(n,p)=1.
\]
When $k$ is algebraically closed, there exists a trace map $H^{2d}%
(X_{\text{\'{e}t}},\mathbb{Z}{}(d))\rightarrow\mathbb{Z}{}$ compatible (via
(A)$_{n}$) with the usual trace map in \'{e}tale cohomology.
\end{description}

\subsection{Values of zeta functions}

Throughout this section, $X$ is a smooth projective variety over a finite
field $k$ with $q$ elements, $r$ is an integer, and $\rho_{r}$ is the rank of
the group of numerical equivalence classes of algebraic cycles of codimension
$r$ on $X$.

We list the following conjectures for reference.

\begin{description}
\item[$T^{r}(X)$] (Tate conjecture): The order of the pole of the zeta
function $Z(X,t)$ at $t=q^{-r}$ is equal to $\rho_{r}$.

\item[$T^{r}(X,l)$] ($l$-Tate conjecture): The map $\mathrm{CH}^{r}%
(X)\otimes\mathbb{Q}{}_{l}\rightarrow H^{2r}(\bar{X}_{\text{\'{e}t}%
},\mathbb{Q}{}_{l}(r))^{\Gamma}$ is surjective.

\item[$S^{r}(X,l)$] (semisimplicity at $1)$: The map $H^{2r}(\bar
{X}_{\text{\'{e}t}},\mathbb{Q}{}_{l}(r))^{\Gamma}\rightarrow H^{2r}(\bar
{X}_{\text{\'{e}t}},\mathbb{Q}{}_{l}(r))_{\Gamma}$ induced by the identity map
is bijective.
\end{description}

\noindent The statement $T^{r}(X)$ is implied by the conjunction of
$T^{r}(X,l)$, $T^{d-r}(X,l)$, and $S^{r}(X,l)$ for a single $l$, and implies
$T^{r}(X,l)$, $T^{d-r}(X,l)$, $S^{r}(X,l)$, $S^{d-r}(X,l)$ for all $l$ (see
\cite{tate1994}, 2.9; \cite{milne2007aim}, 1.11).

Let $V$ be a variety over a finite field $k$. To give a sheaf on
$V_{\text{\'{e}t}}$ is the same as giving a sheaf on $\bar{V}_{\text{\'{e}t}}$
together with a continuous action of $\Gamma\overset{\textup{{\tiny def}}}%
{=}\Gal(\bar{k}/k)$. Let $\Gamma_{0}$ be the subgroup of $\Gamma$ generated by
the Frobenius element (so $\Gamma_{0}\simeq\mathbb{Z}{}$). The Weil-\'{e}tale
topology is defined so that to give a sheaf on $V_{\text{W\'{e}t}}$ is the
same as giving a sheaf on $\bar{V}_{\text{\'{e}t}}$ together with an action of
$\Gamma_{0}$ (\cite{lichtenbaum2005}). For example, for $V=\Spec k$, the
sheaves on $V_{\text{\'{e}t}}$ are the discrete $\Gamma$-modules, and the
sheaves on $V_{\text{W\'{e}t}}$ are the $\Gamma_{0}$-modules. In the
Weil-\'{e}tale topology, the Hochschild-Serre spectral sequence becomes
\begin{equation}
H^{i}(\Gamma_{0},H^{j}(\bar{V}_{\text{\'{e}t}},F))\implies H^{i+j}%
(V_{\text{W\'{e}t}},F). \label{e10}%
\end{equation}
Since
\begin{equation}
H^{i}(\Gamma_{0},M)=M^{\Gamma_{0}},\,M_{\Gamma_{0}},\,0,\,0,\,\ldots
\ \text{for }i=0,1,2,3,\ldots, \label{e11}%
\end{equation}
this gives exact sequences%
\[
0\rightarrow H^{i-1}(\bar{V}_{\text{\'{e}t}},F)_{\Gamma_{0}}\rightarrow
H^{i}(V_{\text{W\'{e}t}},F)\rightarrow H^{i}(\bar{V}_{\text{\'{e}t}%
},F)^{\Gamma_{0}}\rightarrow0,\quad\text{all }i\geq0.
\]
If $F$ is a sheaf on $V_{\text{\'{e}t}}$ such that the groups $H^{j}(\bar
{V}_{\text{\'{e}t}},F)$ are torsion, then the Hochschild-Serre spectral
sequence for the \'{e}tale topology gives exact sequences%
\[
0\rightarrow H^{i-1}(\bar{V}_{\text{\'{e}t}},F)_{\Gamma}\rightarrow
H^{i}(V_{\text{\'{e}t}},F)\rightarrow H^{i}(\bar{V}_{\text{\'{e}t}}%
,F)^{\Gamma}\rightarrow0,\quad\text{all }i\geq0.
\]
The two spectral sequences are compatible, and so, for such a sheaf $F$, the
canonical maps $H^{i}(V_{\text{\'{e}t}},F)\rightarrow H^{i}(V_{\text{W\'{e}t}%
},F)$ are isomorphisms.

Let $X$ be a smooth projective variety over a finite field, and let%
\[
e^{2r}\colon H^{2r}(X_{\text{W\'{e}t}},\mathbb{Z}{}(r))\rightarrow
H^{2r+1}(X_{\text{W\'{e}t}},\mathbb{Z}{}(r))
\]
denote cup-product with the canonical element of $H^{1}(\Gamma_{0}%
,\mathbb{Z}{})=H^{1}(k_{\mathrm{Wet}},\mathbb{Z}{})$, and let%
\[
\chi(X_{\text{W\'{e}t}},\mathbb{Z}(r))=\prod\nolimits_{i\neq2r,2r+1}%
[H^{i}(X_{\text{W\'{e}t}},\mathbb{Z}{}(r))]^{(-1)^{i}}z(e^{2r})
\]
when all terms are defined and finite. Let%
\[
\chi(X,\mathcal{O}{}_{X},r)=\sum\nolimits_{i\leq r,\,j}(-1)^{i+j}(r-i)\dim
H^{j}(X,\Omega_{X/k}^{i})\text{.}%
\]
We define $\chi^{\prime}(X_{\text{W\'{e}t}},\mathbb{Z}(r))$ as for
$\chi(X_{\text{W\'{e}t}},\mathbb{Z}(r))$, but with each group $H^{i}%
(X_{\text{W\'{e}t}},\mathbb{Z}{}(r))$ replaced by its quotient
\[
H^{i}(X_{\text{W\'{e}t}},\mathbb{Z}{}(r))^{\prime}\overset{\textup{{\tiny def}%
}}{=}\frac{H^{i}(X_{\text{W\'{e}t}},\mathbb{Z}{}(r))}{U(H^{i}%
(X_{\text{W\'{e}t}},\mathbb{Z}{}(r)))}.
\]

\begin{theorem}
\label{z7}Let $X$ be a smooth projective variety over a finite field such that
the Tate conjecture $T^{r}(X)$ is true for some integer $r\geq0$. Then
$\chi^{\prime}(X_{\text{W\'{e}t}},\mathbb{Z}{}(r))$ is defined, and%
\begin{equation}
\lim_{t\rightarrow q^{-r}}Z(X,t)\cdot(1-q^{r}t)^{\rho_{r}}=\pm\chi^{\prime
}(X_{\text{W\'{e}t}},\mathbb{Z}(r))\cdot q^{\chi(X,\mathcal{O}{}_{X}%
,r)}.\label{e6}%
\end{equation}
In particular, the groups $H^{i}(X_{\text{W\'{e}t}},\mathbb{Z}{}(r))^{\prime}$
are finite for $i\neq2r,2r+1$. For $i=2r,2r+1$, they are finitely generated.
For all $i$, $U(H^{i}(X_{\text{W\'{e}t}},\mathbb{Z}{}(r)))$ is uniquely divisible.
\end{theorem}

\begin{proof}
We begin with a brief review of \cite{milne1986v}. For an integer
$n=n_{0}p^{s}$ with $\mathrm{gcd}(p,n_{0})=1$,%
\begin{align*}
H^{i}(X_{\text{\'{e}t}},(\mathbb{Z}{}/n\mathbb{Z}{})(r))  &  \overset
{\textup{{\tiny def}}}{=}H^{i}(X_{\text{\'{e}t}},\mu_{n_{0}}^{\otimes
r})\times H^{i-r}(X_{\text{\'{e}t}},\nu_{s}(r))\text{, and}\\
\newline H^{i}(X_{\text{\'{e}t}},\widehat{\mathbb{Z}}(r))  &  \overset
{\textup{{\tiny def}}}{=}\varprojlim\nolimits_{n}H^{i}(X_{\text{\'{e}t}%
},\mathbb{(\mathbb{Z}{}}/n\mathbb{Z}{})(r))
\end{align*}
(ibid. p.\,309). Let%
\[
\epsilon^{2r}\colon H^{2r}(X_{\text{\'{e}t}},\widehat{\mathbb{Z}}%
{}(r))\rightarrow H^{2r+1}(X_{\text{\'{e}t}},\widehat{\mathbb{Z}}{}(r))
\]
denote cup-product with the canonical element of $H^{1}(\Gamma,\widehat
{\mathbb{Z}}{})\simeq H^{1}(k_{\text{\'{e}t}},\widehat{\mathbb{Z}}{})$, and
let
\[
\chi(X,\widehat{\mathbb{Z}}(r))\overset{\textup{{\tiny def}}}{=}%
\prod\nolimits_{i\neq2r,2r+1}[H^{i}(X_{\text{\'{e}t}},\widehat{\mathbb{Z}}%
{}(r))]^{(-1)^{i}}z(\epsilon^{2r})
\]
when all terms are defined and finite (ibid. p.298). Theorem 0.1 (ibid. p.298)
states that $\chi(X,\widehat{\mathbb{Z}}(r))$ is defined if and only if
$S^{r}(X,l)$ holds for all $l$, in which case%
\begin{equation}
\lim_{t\rightarrow q^{-r}}Z(X,t)\cdot(1-q^{r}t)^{\rho_{r}}=\pm\chi
(X,\widehat{\mathbb{Z}}(r))\cdot q^{\chi(X,\mathcal{O}{}_{X},r)}. \label{e7}%
\end{equation}
In particular, if $S^{r}(X,l)$ holds for all $l$, then the groups
$H^{i}(X_{\text{\'{e}t}},\widehat{\mathbb{Z}}(r))$ are finite for all
$i\neq2r$, $2r+1$.

For each $n\geq1$ and $i\geq0$, property (A) of $\mathbb{Z}{}(r)$ gives us an
exact sequence%
\[
0\rightarrow H^{i}(X_{\text{W\'{e}t}},\mathbb{Z}{}(r))^{(n)}\rightarrow
H^{i}(X_{\text{\'{e}t}},(\mathbb{Z}{}/n\mathbb{Z}{})(r))\rightarrow
H^{i+1}(X_{\text{W\'{e}t}},\mathbb{Z}{}(r))_{n}\rightarrow0.
\]
The middle term is finite, and so $H^{i}(X_{\text{W\'{e}t}},\mathbb{Z}%
{}(r))^{(n)}$ is finite for all $i$ and $n$. On passing to the inverse limit,
we obtain an exact sequence%
\begin{equation}
0\rightarrow H^{i}(X_{\text{W\'{e}t}},\mathbb{Z}{}(r))\symbol{94}\rightarrow
H^{i}(X_{\text{\'{e}t}},\widehat{\mathbb{Z}}{}(r))\rightarrow TH^{i+1}%
(X_{\text{W\'{e}t}},\mathbb{Z}{}(r))\rightarrow0 \label{e5}%
\end{equation}
in which the middle term is finite for $i\neq2r,2r+1$. As $TH^{i+1}%
(X_{\text{W\'{e}t}},\mathbb{Z}{}(r))$ is torsion-free, it must be zero for
$i\neq2r,2r+1$. In other words, $TH^{i}(X_{\text{W\'{e}t}},\mathbb{Z}{}(r))=0$
for $i\neq2r+1,2r+2.$

So far we have used only conjecture $S^{r}(X,l)$ (all $l$) and property (A) of
$\mathbb{Z}{}(r)$. To continue, we need to use $T^{r}(X,l)$ (all $l$) and the
property (B) of $\mathbb{Z}{}(r)$. The $l$-Tate conjecture $T^{r}(X,l)$ for
all $l$ implies that the cokernel of the map $\mathrm{CH}^{r}(X)\otimes
_{\mathbb{Z}}\widehat{\mathbb{Z}}{}\rightarrow H^{2r}(X_{\text{\'{e}t}%
},\widehat{\mathbb{Z}}{}(r))$ is torsion. As this map factors through
$H^{2r}(X_{\text{W\'{e}t}},\mathbb{Z}{}(r))\symbol{94}$, it follows that
$TH^{2r+1}(X_{\text{W\'{e}t}},\mathbb{Z}{}(r))=0$ and $H^{2r}%
(X_{\text{W\'{e}t}},\mathbb{Z}{}(r))\symbol{94}\simeq H^{2r}(X_{\text{\'{e}t}%
},\widehat{\mathbb{Z}}{}(r))$. Consider the commutative diagram%
\[
\begin{tikzcd}
H^{2r}(X_{\text{W\'et}},\mathbb{Z}{}(r))\symbol{94}
\arrow{r}{\simeq}\arrow{d}{\widehat{e^{2r}}}
&H^{2r}(X_{\text{\'et}},\widehat{\mathbb{Z}}{}(r))\arrow{d}{\epsilon^{2r}}\\
H^{2r+1}(X_{\text{W\'et}},\mathbb{Z}{}(r))\symbol{94}\arrow{r}
&H^{2r+1}(X_{\text{\'et}},\widehat{\mathbb{Z}}{}(r)).
\end{tikzcd}
\]
As $\epsilon^{2r}$ has finite cokernel, so does the bottom arrow, and so
$TH^{2r+2}(X_{\text{W\'{e}t}},\mathbb{Z}{}(r))=0$. We have now shown that
\[
TH^{i}(X_{\text{W\'{e}t}},\mathbb{Z}{}(r))=0\text{ for all }i
\]
and so ((\ref{e5}) and Corollary \ref{z4})
\[
\left\{
\begin{array}
[c]{l}%
H^{i}(X_{\text{W\'{e}t}},\mathbb{Z}{}(r))\symbol{94}\simeq H^{i}%
(X_{\text{\'{e}t}},\widehat{\mathbb{Z}}{}(r))\\
U(H^{i}(X_{\text{W\'{e}t}},\mathbb{Z}{}(r))\text{ is uniquely divisible}%
\end{array}
\right.  \text{ for all }i.
\]
In particular, we have proved the first statement of the theorem except that
each group $H^{i}(X_{\text{W\'{e}t}},\mathbb{Z}{}(r))^{\prime}$ has been
replaced by its completion. It remains to prove that $H^{i}(X_{\text{W\'{e}t}%
},\mathbb{Z}{}(r))^{\prime}$ is finite for $i\neq2r,2r+1$ and is finitely
generated for $i=2r,2r+1$ (for then $H^{i}(X_{\text{W\'{e}t}},\mathbb{Z}%
{}(r))\symbol{94}\simeq H^{i}(X_{\text{W\'{e}t}},\mathbb{Z}{}(r))^{\prime
}\otimes\widehat{\mathbb{Z}}$).

The kernel of $H^{i}(X_{\text{W\'{e}t}},\mathbb{Z}{}(r))^{\prime}%
\rightarrow\left(  H^{i}(X_{\text{W\'{e}t}},\mathbb{Z}{}(r))^{\prime}\right)
\symbol{94}$ is $U(H^{i}(X_{\text{W\'{e}t}},\mathbb{Z}{}(r))^{\prime})=0$, and
so $H^{i}(X_{\text{W\'{e}t}},\mathbb{Z}{}(r))^{\prime}$ is finite for
$i\neq2r,2r+1$.

It remains to show that the groups $H^{2r}(X_{\text{W\'{e}t}},\mathbb{Z}%
{}(r))^{\prime}$ and $H^{2r+1}(X_{\text{W\'{e}t}},\mathbb{Z}{}(r))^{\prime}$
are finitely generated. For this we shall need property (C) of $\mathbb{Z}%
{}(r)$. For a fixed prime $l\neq p$, the pairings in (C) give rise to a
commutative diagram%
\[
\begin{tikzpicture}[text height=1.5ex, text depth=0.25ex]
\node (a) at (0,0) {$H^{2r}(X_{\text{W\'{e}t}},\mathbb{Z}{}(r))^{\prime}/\mathrm{tors}$};
\node (t) at (2.2,0) {$\times$};
\node (b) at (5,0) {$H^{2d-2r+1}(X_{\text{W\'{e}t}},\mathbb{Z}{}(d-r))^{\prime}/\mathrm{tors}$};
\node (c) at (9,0) {$\mathbb{Z}$};
\node (d) [below=of a] {$H^{2r}(X_{\text{\'{e}t}},\mathbb{Z}_{l}{}(r))/\mathrm{tors}$};
\node [below=of t] {$\times$};
\node (e) [below=of b] {$H^{2d-2r+1}(X_{\text{\'{e}t}},\mathbb{Z}_{l}{}(d-r))/\mathrm{tors}$};
\node (f) [below=of c] {$\mathbb{Z}{}_{l}$};
\draw[->,font=\scriptsize,>=angle 90]
(a) edge (d)
(b) edge (e)
(c) edge (f)
(b) edge (c)
(e) edge (f);
\end{tikzpicture}
\]
to which we wish to apply Proposition \ref{z6}. The bottom pairing is
nondegenerate, the group $U(H^{2r}(X_{\text{W\'{e}t}},\mathbb{Z}{}%
(r))^{\prime})$ is zero, and the group $H^{2d-2r+1}(X_{\text{W\'{e}t}%
},\mathbb{Z}{}(d-r))^{(l)}$ is finite, and so the proposition shows that
$H^{2r}(X_{\text{W\'{e}t}},\mathbb{Z}{}(r))^{\prime}/\mathrm{tors}$ is
finitely generated. Because $U(H^{2r}(X_{\text{W\'{e}t}},\mathbb{Z}%
{}(r))^{\prime})=0$, the torsion subgroup of $H^{2r}(X_{\text{W\'{e}t}%
},\mathbb{Z}{}(r))^{\prime}$ injects into the torsion subgroup of
$H^{2r}(X_{\text{\'{e}t}},\widehat{\mathbb{Z}}(r))$, which is finite
(\cite{gabber1983}). Hence $H^{2r}(X_{\text{W\'{e}t}},\mathbb{Z}{}%
(r))^{\prime}$ is finitely generated. The group $H^{2r+1}(X_{\text{W\'{e}t}%
},\mathbb{Z}{}(r))^{\prime}$ can be treated similarly.
\end{proof}

\begin{remark}
\label{z8}In the proof, we didn't use the full force of $T^{r}(X)$.
\end{remark}

We shall need the following standard result.

\begin{lemma}
\label{z9}Let $A$ be a (noncommutative) ring and let $\bar{A}$ be the quotient
of $A$ by a nil ideal $I$ (i.e., a two-sided ideal in which every element is
nilpotent). Then:

\begin{enumerate}
\item an element of $A$ is invertible if it maps to an invertible element of
$\bar{A}$;

\item every idempotent in $\bar{A}$ lifts to an idempotent in $A$, and any two
liftings are conjugate by an element of $A$ lying over $1_{\bar{A}}$;

\item let $a\in A$; every decomposition of $\bar{a}$ into a sum of orthogonal
idempotents in $\bar{A}$ lifts to a similar decomposition of $a$ in $A$.
\end{enumerate}
\end{lemma}

\begin{nt}
We denote $a+I$ by $\bar{a}$.

(a) It suffices to consider an element $a$ such that $\bar{a}=1_{\bar{A}}$.
Then $(1-a)^{N}=0$ for some $N>0$, and so%
\[
\overbrace{\left(  1-(1-a)\right)  }^{\displaystyle{a}}\left(
1+(1-a)+(1-a)^{2}+\cdots+(1-a)^{N-1}\right)  =1.
\]

(b) Let $a$ be an element of $A$ such that $\bar{a}$ is idempotent. Then
$(a-a^{2})^{N}=0$ for some $N>0$, and we let $a^{\prime}=$ $(1-(1-a)^{N})^{N}%
$. A direct calculation shows that $a^{\prime}a^{\prime}=a^{\prime}$ and that
$\bar{a}^{\prime}=\bar{a}$.

Let $e$ and $e^{\prime}$ be idempotents in $A$ such that $\bar{e}=\bar
{e}^{\prime}$. Then $a\overset{\textup{{\tiny def}}}{=}e^{\prime
}e+(1-e^{\prime})(1-e)$ lies above $1_{\bar{A}}$ and satisfies $e^{\prime
}a=e^{\prime}e=ae$.

(c) Follows easily from (b).
\end{nt}

\begin{proposition}
\label{z10}Let $X$ be a smooth projective variety over a finite field $k$, and
let $r$ be an integer${}$. Assume that for some prime $l$ the ideal of
$l$-homologically trivial correspondences in $\mathrm{CH}^{\dim X}(X\times
X)_{\mathbb{Q}{}}$ is nil. Then $H^{i}(X_{\mathrm{et}},\mathbb{Z}{}(r))$ is
torsion for all $i\neq2r$, and the Tate conjecture $T^{r}(X)$ implies that
$H^{2r}(X_{\mathrm{et}},\mathbb{Z}{}(r))$ is finitely generated modulo torsion.
\end{proposition}

\begin{proof}
This is essentially proved in \cite{jannsen2007}, pp.\thinspace131--132, and
so we only sketch the argument. Set $d=\dim X$ and let $k=\mathbb{F}{}_{q}$.

According to Lemma \ref{z9}, there exist orthogonal idempotents $\pi
_{0},\ldots,\pi_{2d}$ in $\mathrm{CH}^{\dim X}(X\times X)_{\mathbb{Q}{}}$
lifting the K\"{u}nneth components of the diagonal in the $l$-adic topology.
Let $h^{i}X=(hX,\pi_{i})$ in the category of Chow motives over $k$. Let
$P_{i}(T)$ denote the characteristic polynomial $\det(T-\varpi_{X}|H^{i}%
(\bar{X}_{\mathrm{et}},\mathbb{Q}{}_{l})$ of the Frobenius endomorphism
$\varpi_{X}$ of $X$ acting on $H^{i}(\bar{X}_{\mathrm{et}},\mathbb{Q}{}_{l})$.
Then $P_{i}(\varpi_{X})$ acts as zero on the homological motive $h_{\text{hom}%
}^{i}X$, and so $P_{i}(\varpi_{X})^{N}$ acts as zero on $h^{i}X$ for some
$N\geq1$ (from the nil hypothesis). We shall need one last property of Bloch's
complex, namely, that $H^{i}(X_{\mathrm{W\acute{e}t}},\mathbb{Z}%
(r))_{\mathbb{Q}{}}\simeq K_{2r-i}(X)^{(r)}$ where $K_{2r-i}(X)^{(r)}$ is the
subspace of $K_{2r-i}(X)_{\mathbb{Q}{}}$ on which the $n$th Adams operator
acts as $n^{r}$ for all $r$. 

The $q$th Adams operator acts as the Frobenius operator, and so $\varpi_{X}$
acts as multiplication by $q^{r}$ on $K_{2r-i}(X)^{(r)}$. Therefore
$H^{i}(X_{\mathrm{W\acute{e}t}},\mathbb{Z}{}(r))_{\mathbb{Q}{}}$ is killed by
$P_{i}(q^{r})^{N}$, which is nonzero for $i\neq2r$ (by the Weil conjectures),
and so $H^{i}(X_{\mathrm{W\acute{e}t}},\mathbb{Z}{}(r))$ is torsion for
$i\neq2r$.

The Tate conjecture implies that $P_{2r}(T)=Q(T)\cdot(T-q^{r})^{\rho_{r}}$
where $Q(q^{r})\neq0$, and so%
\[
1=q(T)Q(T)^{N}+p(T)(T-q^{r})^{N\rho_{r}},\quad\text{some }q(T)\text{, }%
p(T)\in\mathbb{Q}{}[T].
\]
As before, $P_{2r}(\omega_{X})^{N}$ acts as zero on $h^{2r}X$ for some
$N\geq1$. Therefore $q(\varpi_{X})Q(\varpi_{X})^{N}$ and $p(\varpi_{X}%
)(\varpi_{X}-q^{r})^{N\rho_{r}}$ are orthogonal idempotents in $\End(h^{2r}X)$
with sum $1$, and correspondingly $h^{2r}X=M_{1}\oplus M_{2}$. Now
$H^{2r}(M_{1},\mathbb{\mathbb{Z}{}}{}(r))_{\mathbb{Q}{}}=0$ because
$Q(\varpi_{X})^{N}$ is zero on $M_{1}$ and $Q(q^{r})\neq0$. On the other hand,
$M_{2}$ is isogenous to $(\mathbb{L}{}^{\otimes r})^{\rho_{r}}$ where
$\mathbb{L}{}$ is the Lefschetz motive (\cite{jannsen2007}, p.\thinspace132),
and so $H^{2r}(M_{2},\mathbb{Z}{}(r))$ differs from
\[
H^{2r}(\mathbb{L}{}^{\otimes r},\mathbb{Z}{}(r))^{\rho_{r}}\simeq
H^{2r}(\mathbb{P}{}^{d},\mathbb{Z}{}(r))^{\rho_{r}}\simeq\mathbb{Z}{}%
^{\rho_{r}}%
\]
by a torsion group.
\end{proof}

\begin{nt}
When $k=\mathbb{F}{}_{q}$, the $q$th Adams operator acts as $\varpi$ (Hiller
1981, \S 5; Soul\'{e} 1985, 8.1), and so $K_{i}(X)^{(j)}$ is the subspace on
which $\varpi$ acts as $q^{j}$ (because the $m^{i}$-eigenspace of the $m$th
Adams operators is independent of $m$, Seiler 1988, Theorem 1).
\end{nt}

\begin{theorem}
\label{z11}Let $X$ be a smooth projective variety over a finite field such
that the Tate conjecture $T^{r}(X)$ is true for some integer $r\geq0$. Assume
that, for some prime $l$, the ideal of $l$-homologically trivial
correspondences in $\mathrm{CH}^{\dim X}(X\times X)_{\mathbb{Q}{}}$ is nil.
Then $\chi(X_{\text{W\'{e}t}},\mathbb{Z}{}(r))$ is defined, and%
\begin{equation}
\lim_{t\rightarrow q^{-r}}Z(X,t)\cdot(1-q^{r}t)^{\rho_{r}}=\pm\chi
(X_{\text{W\'{e}t}},\mathbb{Z}(r))\cdot q^{\chi(X,\mathcal{O}{}_{X}%
,r)}.\label{e8}%
\end{equation}
In particular, the groups $H^{i}(X_{\text{W\'{e}t}},\mathbb{Z}{}(r))$ are
finite for $i\neq2r,2r+1$. For $i=2r,2r+1$, they are finitely generated.
\end{theorem}

\begin{proof}
This will follow from Theorem \ref{z7} once we show that the groups
$U^{i}\overset{\textup{{\tiny def}}}{=}U(H^{i}(X_{\text{W\'{e}t}},\mathbb{Z}%
{}(r)))$ are zero. Because $H^{i}(X_{\text{W\'{e}t}},\mathbb{Z}{}(r))$ is
finitely generated modulo torsion (Proposition \ref{z10}), it does not contain
a nonzero $\mathbb{Q}{}$-vector space, and so $U^{i}=0$ (Corollary \ref{z4}).
\end{proof}

\begin{remark}
\label{z12} For a smooth projective algebraic variety $X$ whose Chow motive is
finite-dimensional, the ideal of $l$-homologically trivial correspondences in
$\mathrm{CH}^{\dim X}(X\times X)_{\mathbb{Q}{}}$ is nil for all prime $l$
(Kimura). It is conjectured (Kimura and O'Sullivan) that the Chow motives of
algebraic varieties are always finite-dimensional, and this is known for those
in the category generated by the motives of abelian varieties. On the other
hand, Beilinson has conjectured that, over finite fields, rational equivalence
with $\mathbb{Q}{}$-coefficients coincides with with numerical equivalence,
which implies that the ideal in question is always null (not merely nil).
\end{remark}

{
\bibliographystyle{cbe}
\bibliography{D:/Current/refs}
}

\end{document}